\documentclass[12pt]{amsart}

\usepackage{amsfonts}
\usepackage{amsmath}
\usepackage{amssymb}
\usepackage{mathrsfs}
\usepackage{amscd}
\usepackage[all]{xy}
\usepackage[pdftex,final]{graphicx}
\input cyracc.def

\newtheorem{lem}{Lemma}[section]
\newtheorem{thm}[lem]{Theorem}

\newtheorem{cor}[lem]{Corollary}
\theoremstyle{definition}

\newtheorem{exa}[lem]{Example}
\newtheorem{rem}[lem]{Remark}

\title{
Tate Kernels, \'Etale $K$-theory  
and the Gross Kernel}
\author{Kevin Hutchinson}
\address{Department of Mathematics\\
University College Dublin\\ 
Belfield\\
Dublin 4\\ 
IRELAND}
\email{kevin.hutchinson@ucd.ie}
\date{\today}


\newcommand{\imp}{\Longrightarrow}

\newcommand{\Q}{\Bbb{Q}}
\newcommand{\F}[1]{\Bbb{F}_{#1}}
\newcommand{\E}{\mathcal{E}}
\newcommand{\Z}{\Bbb{Z}}

\newcommand{\K}{\mathcal{K}}
\newcommand{\M}{\mathcal{M}}
\newcommand{\NN}{{\mathcal{N}}}
\newcommand{\NF}{\tilde{N}}

\newcommand{\XX}{\tilde{X}}
\newcommand{\LL}{\tilde{H}}

\newcommand{\ed}[2]{Q_{#1,#2}}

\newcommand{\h}[1]{\tilde{H}_{#1}}

\renewcommand{\dim}[2]{\mathrm{dim}_{#1}\left(#2\right)}
\newcommand{\<}{\langle}
\renewcommand{\>}{\rangle}

\newcommand{\set}[2]{\{#1 \mid #2 \}}

\newcommand{\psyl}[2]{#1\{ #2\}}

\newcommand{\dvv}[1]{\mathrm{Div}(#1)}
\newcommand{\dual}[1]{{#1}^*}
\newcommand{\pdual}[1]{{#1}^{*}}
\newcommand{\Pdual}[1]{\pdual{(#1)}}

\newcommand{\colim}{\mathrm{colim}}

\renewcommand{\ker}[1]{\mathrm{Ker}(#1)}
\newcommand{\image}[1]{\mathrm{Im}(#1)}

\newcommand{\coker}[1]{\mathrm{Coker}(#1)}
\newcommand{\edm}[1]{\mathrm{End}(#1)}

\renewcommand{\hom}[3]{\mathrm{Hom}_{#1}(#2,#3)}
\newcommand{\aut}[1]{\mathrm{Aut}(#1)}

\newcommand{\units}[1]{#1^{\times}}

\newcommand{\free}[2]{\mathrm{Fr}_{#1}\left( #2 \right)}

\newcommand{\zero}{\{ 0 \}}
\newcommand{\ab}[1]{#1^{\mbox{\tiny ab}}}

\newcommand{\trs}[2]{\mathrm{tor}_{#1}(#2)}
\newcommand{\trlx}{\trs{\Lambda}{X}}

\newcommand{\ntr}[1]{\mathcal{O}_{#1}}
\newcommand{\sntr}[2]{\mathcal{O}_{#2}^{#1}}

\newcommand{\proots}[1]{\mu_{p^\infty}(#1)}

\newcommand{\artin}[3]{\left(\frac{#1,#2}{#3}\right)}

\newcommand{\gal}[2]{\mathrm{Gal}(#1 /#2)}

\newcommand{\cl}[1]{\mathrm{Cl}(#1)}
\newcommand{\cls}[2]{\mathrm{Cl}^{#1}(#2)}
\newcommand{\as}[1]{A^S(#1)}

\newcommand{\un}[2]{\mathrm{U}^{#1}_{#2}}
\newcommand{\norm}[2]{\mathrm{N}_{#1/#2}}
\newcommand{\ram}[2]{S_{\mbox{\tiny ram}}(#1/#2)}

\newcommand{\pun}[1]{\tilde{\mathcal{U}}_{#1}}

\newcommand{\logcls}[1]{\tilde{Cl}_{#1}}
\newcommand{\ldiv}[1]{\tilde{Dl}_{#1}}

\newcommand{\gk}[1]{\tilde{\mathcal{E}}_{#1}}
\newcommand{\gkk}[1]{\mathcal{G}_{#1}}

\newcommand{\rk}[1]{\mathcal{R}_{#1}}
\newcommand{\dl}[1]{\tilde{Dl}_{#1}}
\newcommand{\kf}[1]{K_2 (#1)}

\newcommand{\ket}[3]{K_{#1}^{\mbox{\tiny{\'et}}}(\sntr{#2}{#3})}

\newcommand{\kos}[2]{K_{#1} (\ntr{#2}^S)}

\newcommand{\tatk}[2]{D^{(#2)}_{#1}}
\newcommand{\gtatk}[3]{D^{[#2,#3]}_{#1}}
\newcommand{\ntatk}[3]{D^{(#2,#3)}_{#1}}

\newcommand{\coh}[3]{\mathrm{H}^{#1}(#2,#3 )}
\newcommand{\coinv}[2]{\left( #2\right)_{#1}}
\newcommand{\Inv}[2]{\left( #2\right)^{#1}}
\newcommand{\inv}[2]{#2^{#1}}

\begin{document}

\maketitle


\begin{abstract}
For an odd prime $p$ and a number field $F$ containing a $p$th root of
unity, we study generalised Tate kernels, $\gtatk{F}{i}{n}$, for
$i\in \Z$ and $n\geq 1$,  having  the properties that if 
$i\geq 2$ and if either $p$ does not divide $i$ or $\mu_{p^n}\subset
F$ then there are natural isomorphisms
$\gtatk{F}{i}{n}\cong\ket{2i-1}{S}{F}/p^n$, and that they  are periodic
modulo a power of $p$ which depends on $F$ and $n$. Our main result is
 that if
the Gross-Jaulent conjecture holds for $(F,p)$ then there is a natural
isomorphism $\gtatk{F}{1}{n}\cong\gk{F}/p^n$ where $\gk{F}$ is the
Gross kernel. We apply this result to compute lower bounds for
capitulation kernels in even \'etale $K$-theory.
\end{abstract}

\footnote{The research in this article was partially funded by Irish Research Council for Science, 
Engineering and Technology Basic Research Grant SC/02/265.
}

\section{Introduction}

Let $F$ be an algebraic number field containing the $p$-th roots of
unity $\mu_p$ for some  odd prime number $p$. Let
$S$ be the set of $p$-adic  primes
of $F$ and let $\sntr{S}{F}$ be the ring of $S$-integers in $F$. 

Let $E/F$ be a Galois extension with group $G$. The goal of this paper
is to better understand and to calculate lower bounds for the $K$-theory
capitulation kernels $\ker{f_i}$ where
$f_i:\ket{2i-2}{S}{F}\to\ket{2i-2}{S}{E}^G$  is the natural functorial
map. Our main results apply to the case where $G$ is cyclic of degree
$p^n$.

We follow the strategy established by Assim and Movahhedi in
\cite{am:cap}. Generalising a result of Kahn, \cite{kah:des}, they
observe that, when $G$ is a cyclic $p$-group, $\ker{f_i}$ and $\coker{f_i}$ have the
same order and, in the case where $|G|=p$, they describe this cokernel
in terms of certain `generalised Tate kernels' $\tatk{F}{i}$ which are
subgroups of $F^\times$.  $\tatk{F}{2}$ is the classical Tate kernel;
i.e the group $C_F=\{ a\in \units{F}| \{ a,\zeta_p\}=0\mbox{ in }
K_2(F)\}$. Furthermore, Greenberg has shown that 
if $\tilde{F}$  denotes the compositum of the
$\Z_p$-extensions of $F$ and if $A_F=\{ a\in \units{F}|\
\sqrt[p]{a}\in \tilde{F}\}$, then $\tatk{F}{0}\subset A_F$ with
equality if and only if Leopoldt's conjecture holds for $(F,p)$.   

In this paper we describe `Tate kernels' $\gtatk{F}{i}{n}$ for all
$i\in\Z$ with
the following four properties:
\begin{enumerate}
\item \label{tate}
$\gtatk{F}{i}{1}=\tatk{F}{i}/(\units{F})^p(i-1)$ where $M(j)$
  denotes the $j$th Tate twist of the Galois module $M$.

\item \label{etale}
There are natural isomorphisms 
\[
\gtatk{F}{i}{n}\cong \ket{2i-1}{S}{F}/p^n
\]
for $i\geq 2$ whenever either $p$ does not divide $i$ or $F$ contains
the $p^n$th roots of unity.

\item \label{gross}
If the Gross conjecture holds for $(F,p)$ then
\[
\gtatk{F}{1}{n}\cong \gk{F}/p^n
\] 
where $\gk{F}$ is the Gross kernel of $F$. 

\item \label{period}
The Tate kernels have a natural periodicity property: 

There exists a
  number $m=m_n(F)\geq 0$ such that 
\[
\gtatk{F}{j}{n}=\gtatk{F}{i}{n}(j-i)
\]
whenever $i\equiv j\pmod{p^m}$. Furthermore $m=0$ if $\mu_{p^r}\subset
F$ for sufficiently large $r$.
\end{enumerate}

These groups, and variations on them, have been studied in a number of
places. The central ideas go back to Tate, \cite{tate:galcoh}, and
property (\ref{etale}) (at least for $i=2$) is implicit there. The
periodicity property is essentially found in Greenberg's paper, \cite{gr:zp}, but
see also Assim and Movahhedi, \cite{am:cap}, Lemma 2.1 and Vauclair,
\cite{vauclair:cap}, section 4. The main contribution of the present
paper is property (\ref{gross}), although it should be noted that a key step in the proof is
Theorem 2.3 of Kolster, \cite{kol:idel}.
In particular, (\ref{gross}) implies that if $F$ has only one $p$-adic
prime then $\tatk{F}{1}=U^S_F\cdot(\units{F})^p$. This, in conjunction
with property (\ref{period}), explains the
theorem of Kersten, \cite{ik:zp}, that if $F=\Q(\zeta_{p^r})$ for sufficiently large
$r$, then $A_F= U^S_F\cdot(\units{F})^p=C_F$. Our result in this paper
grew out of
the attempt to understand this theorem of Kersten. 

The layout of the paper is as follows: In section \ref{sec:tate}, we
introduce the  Tate kernels and establish properties
(\ref{tate}), (\ref{gross}) and (\ref{period}). In section \ref{sec:coh}, we
give a cohomological description of  the Tate kernels and use this to
establish property (\ref{etale}), as well as to prove certain basic
algebraic properties which have already been used section
\ref{sec:tate}.

In section \ref{sec:gen} we show how to describe the groups
$\coker{f_i}$, when $E/F$ is a cyclic $p$-extension, in terms of the
Tate kernels and in the last two sections we apply these results to
give lower bounds for $\ker{f_i}$. In particular, in section
\ref{sec:wild}, we deal with cyclic $p$-extensions in which there are
no tamely-ramified primes. We exploit the relationship between the
Gross kernel and the logarithmic class group to show  (Theorem
\ref{thm:cyclo}) that if $E/F$ is a finite layer of
the $p$-cyclotomic extension of $F$ and if the Gross conjecture holds
for $(E,p)$, then for infinitely many $i$,
$|\ker{f_i}|=|\ker{f_1}|$ where $f_1$ is the natural functorial map
$\logcls{F}\to\logcls{E}^G$ of logarithmic class groups. We give
examples both of cyclotomic and of non-cyclotomic extensions for which the
$\ker{f_i}$ are all non-trivial.

\textit{Notation:} For an abelian group $A$, $A[n]$ and $A/n$ denote
the kernel and cokernel respectively of multiplication by $n\in
N$. 
$\dvv{A}$
denotes the maximal divisible subgroup of $A$. 

If $R$ is an integral domain and $M$ is an $R$-module, then
$\trs{R}{M}$ is the torsion submodule of $M$ and
$\free{R}{M}=M/\trs{R}{M}$.

For a field $F$, $G_F$ denotes its Galois group. The Tate module for
$(F,p)$ is the $\Z_p[G_F]$-module $\Z_p(1):=\lim_i\mu_{p^i}$, the
(inverse) limit being taken over the natural surjections
$\mu_{p^{i+1}}\to\mu_{p^i}$, $\zeta\mapsto \zeta^p$. More generally,
for any $i\in\Z$, we have $\Z_p(i):=\Z_p^{\otimes i}$ for $i\geq 1$,
$\Z_p(0):=\Z_p$ and $\Z_p(-i):=\hom{\Z_p}{\Z_p(i)}{\Z_p}$ for $i\geq
1$. For any $\Z_p$-$G_F$-module $M$, the $i$th Tate twist is the
$\Z_p$-$G_F$-module $M(i):= M\otimes_{\Z_p}\Z_p(i)$ (with diagonal
Galois action). Since $\Z_p(i)$ is isomorphic to $\Z_p$ as a
$\Z_p$-module, $M(i)$ is isomorphic to $M$ with a twisted
Galois-action. Observe that if $M=M[p^r]$, then the natural surjection
$\Z_p(i)\to \mu_{p^r}^{\otimes i}$ induces an isomorphism $M(i)\cong
M\otimes   \mu_{p^r}^{\otimes i}$.

By the \emph{Pontryagin dual} of the $\Z_p$-$G_F$-module $M$ we will
mean  the module $\pdual{M}=\hom{\Z_p}{M}{\Q_p/\Z_p}$.

\section{Generalised Tate Kernels}
\label{sec:tate}

Everywhere below $p$ is an odd prime, $F$ is a number field
containing the $p$-th roots of unity, $\mu_p$, and $S=S_p(F)$ is the
set of prime ideals lying over $p$. 
Let $F_n=F(\mu_{p^n})$ and let $F_\infty=\cup_{n=1}^\infty F_n$. Let
$\Gamma=\gal{F_\infty}{F}$. Let $\cls{S}{F}$ denote the $S$-classgroup of $F$ and let $\as{F}$
denote the $p$-Sylow subgroup $\cls{S}{F}\{ p\}$.  Let also 
$
\proots{F}=F\cap\mu_{p^\infty}=\mu_{p^{n_F}}
$
be the group of $p$-power roots of unity in $F$. We will let $F_S/F$ be the
maximal algebraic extension of $F$ unramified outside the $S$ and we
will let $G^S_F:=\gal{F_S}{F}$.

As usual, $\Lambda$ will denote the Iwasawa algebra
$\Z_p[[\Gamma]]:=\lim\Z_p[\Gamma/p^n]$. If we fix a topological
generator, $\gamma_0$, of $\Gamma$, then we get an isomorphism of
topological rings  
$\Lambda\cong \Z_p[[T]]$, $\gamma_0\leftrightarrow 1+T$.

Let $K$ be the maximal abelian pro-$p$ extension of $F_\infty$ and let
$\K = \units{F_\infty}\otimes \Q_p/\Z_p$. Kummer Theory gives a
perfect duality pairing
\begin{eqnarray}\label{eqn:pair}
\left\<\  ,\  \right\>:\K\times \gal{K}{F_\infty}\to \mu_{p^\infty}
\end{eqnarray}
determined by the formula
\[
g(\sqrt[p^n]{a})=\left\< a\otimes\frac{1}{p^n},g\right\>\sqrt[p^n]{a}
\]
for $g\in \gal{K}{F_\infty}$, $a\in \units{F_\infty}$ and $n\geq 1$.
Furthermore, each term is naturally a $\Gamma$-module and this pairing
is compatible with the $\Gamma$-actions in the sense that
$(\gamma(\alpha),\gamma(g))=\gamma((\alpha,g))$ for all
$\gamma\in\Gamma$, $\alpha\in \K$ and $g\in\gal{K}{F_\infty}$.

Observe that for all $i\in\Z$, 
\[
\K(i)= F_\infty^\times\otimes \Q_p/\Z_p(i)=\bigcup_n
F_\infty^\times\otimes \mu_{p^n}^{\otimes i}=
\bigcup_n
F_\infty^\times/(F_\infty^\times)^{p^n}\otimes \mu_{p^n}^{\otimes i}.
\]

Now let $M$ be the maximal abelian pro-$p$ extension of $F_\infty$
which is unramified outside $p$. Let $X=\gal{M}{F_\infty}$ and let
$\M$ be the subgroup of $\K$ corresponding to $X$ under the pairing
(\ref{eqn:pair}); i.e. $\M$ is the orthogonal complement of $\gal{K}{M}$ with
respect to the pairing. There is thus an induced perfect pairing 
$\M\times X\to \mu_{p^\infty}$. Let  $\NN$ be the complement of
$\trlx$, the maximal $\Lambda$-torsion submodule of $X$,  with respect to this latter
pairing. We also let  $\NF/F_\infty$ be the corresponding field
extension (so that $\trlx = \gal{M}{\NF}$). 

For each $i\in\Z$ we define the subgroups $\tatk{F}{i}$ of $\units{F}$:
\[
\tatk{F}{i}=\set{ a\in \units{F} }{ a\otimes \chi \in
  \dvv{\NN(i-1)^\Gamma}\mbox{ for all }\chi \in \mu_p^{\otimes (i-1)}}.
\]

Observe, that if $a\in (\units{F_\infty})^p$, then $a\otimes\chi=1$
for all $\chi\in \mu_p^{\otimes (i-1)}$. Thus
$\proots{F}\cdot(\units{F})^p\subset\tatk{F}{i}$ for all $i$. 

\begin{rem}
In \cite{am:cap}, Assim and Movahhedi introduce the groups
$\tatk{F}{i}$, $i\geq 2$, but define them using $\K$ instead of
$\NN$. However, as they observe, when $i\geq 2$, this defines the same
group. While the $\Lambda$-module $\K$ is, at first glance, a more
natural object, there are advantages to using $\NN$ in the
definition of the groups $\tatk{F}{i}$. In the first place, Iwasawa has proved some very strong
results on the $\Lambda$-module structure of the dual group of $\NN$. 
Furthermore, using
$\NN$ instead of $\K$, we get a sequence of groups with analogous
properties defined for \emph{all} $i\in \Z$, including $i=1$, which is
the main focus of this paper.
\end{rem}

Regarding the structure of $\NN$,  the 
following is known:
 
Let  $\XX=\free{\Lambda}{X}=X/\trlx$ be the dual of $\NN$ with respect to the Kummer
pairing.
Iwasawa, \cite{iwa:zp} Theorem 17, proved that there is a short exact
sequence of $\Lambda$-modules  
\[
0\to \XX \to \Lambda^{r_2}\to \h{F}\to 0
\]
with $\h{F}$ finite. Greenberg, \cite{gr:zp}, pointed out that the
arguments of Iwasawa show that $\h{F}$ is abstractly isomorphic to (in
fact, Kummer dual to
) the group
$\ker{\as{F_n}\to \as{F_\infty}}$ for all sufficiently large $n$. We will use this fact below. (Coates, \cite{coates:k2}, also
showed that $\h{F}$ is abstractly isomorphic to
$\ker{K_2(\ntr{F_n}^S)\to K_2(\ntr{F_\infty}^S)}$ for all sufficiently
large $n$.  For a more precise assertion about limits, see Kahn,
\cite{kah:des}, Th\'eor\`eme 6.2.)

We deduce the following well-known fact:
\begin{lem} 
For all $i\in \Z$,
\[
\dvv{\NN(i)^\Gamma}\cong (\Q_p/\Z_p)^{r_2}.
\]
\end{lem}
\begin{proof}
Taking the $\Gamma$-action into account, the Kummer pairing gives an
isomorphism of $\Lambda$-modules $\NN(i)\cong \pdual{\XX(-1-i)}$. It
follows that 
\[
\NN(i)^\Gamma\cong \Pdual{\XX(-1-i)_\Gamma}=\Pdual{\XX(-1-i)/T\XX(-1-i)}.
\]
Twisting the exact sequence of Iwasawa by $-1-i$, and using the fact
that $\Lambda(t)\cong \Lambda$ for all $t\in \Z$, gives the exact
sequence
\[
0\to \XX(-1-i) \to \Lambda^{r_2}\to \h{F}(-1-i)\to 0.
\]
Thus the natural map $\XX(-1-i)/T\XX(-1-i)\to
\Lambda^{r_2}/T\Lambda^{r_2}=\Z_p^{r_2}$ has finite kernel and
cokernel. Taking Pontryagin duals again, it follows that there is a
map 
$(\Q_p/\Z_p)^{r_2}\to \NN(i)^\Gamma$ with finite kernel and cokernel.
\end{proof}

 Greenberg (\cite{gr:zp}, p1238) conjectured that
$\dvv{\K(t)^\Gamma}=\dvv{\NN(t)^\Gamma}$ for all
$t\in\Z\setminus\zero$. As he pointed out, the analogous statement for
 $t=0$ cannot possibly hold since  
$ \dvv{\inv{\Gamma}{\K}}\supset\units{F}\otimes
\Q_p/\Z_p$ which is a countable direct sum of copies of
$\Q_p/\Z_p$. Note that if the conjecture is true for
any given $t\in \Z$, then it follows that $\tatk{F}{t+1}=
\set{ a\in \units{F} }{ a\otimes \chi \in
  \dvv{\K(t)^\Gamma}\mbox{ for all }\chi \in \mu_p^{\otimes t}}$. 
If we interpret the groups $\K(t)^\Gamma$ in terms of Galois
cohomology (see section \ref{sec:coh} below), Greenberg's conjecture
is equivalent to Schneider's conjecture in \cite{sch:galcoh} that the
groups $\coh{2}{G^S_F}{\Q_p/\Z_p(i)}$ vanish for all $i\not= 1$. This
latter conjecture was proved by Soul\'e (\cite{sou:etcoh}) when $i\geq
2$ using Borel's theorem (\cite{bor:coh}) on the finiteness of the groups $\kos{2i}{F}$.
It follows, as remarked above,  that for $i\geq 2$ we have 
\[
\tatk{F}{i}=
\set{ a\in \units{F} }{ a\otimes \chi \in
  \dvv{\K(i-1)^\Gamma}\mbox{ for all }\chi \in \mu_p^{\otimes (i-1)}}.
\]

Greenberg showed in \cite{gr:zp} that 
$\tatk{F}{2}$ is the classical Tate Kernel of $F$; the group\\
 $C_F:=
\set{a\in\units{F}}{\{ a, \zeta_p\} = 0 \in \kf{F}}$. 

Furthermore, when $i=0$,
Greenberg  showed that 
\[
\set{ a\in \units{F} }{ a\otimes \chi \in
  \dvv{\K(-1)^\Gamma}\mbox{ for all }\chi \in \mu_p^{\otimes (-1)}}=A_F:=
\{ a\in F^\times \mid \sqrt[p]{a}\in \tilde{F_1} \}
\]
 where
$\tilde{F_1}$ is the compositum of the first layers of the
$\Z_p$-extensions of $F$. Thus,  $\tatk{F}{0}\subset A_F$ and
 Greenberg's conjecture in this case is that equality holds here.
This is equivalent to Leopoldt's conjecture for the pair $(F,p)$ since
$\dim{\F{p}}{\tatk{F}{i}/(\units{F})^p}=1+r_2$ (see below).

Our main interest in the groups $\tatk{F}{i}$ in this paper stems from
the work of Assim and Movahhedi (\cite{am:cap}) who show that for $i\geq 2$ that there is a natural
isomorphism  
\[
\tatk{F}{i}/(\units{F})^p(i-1)\cong \ket{2i-1}{S}{F}/p
\]
(for a generalisation, see Corollary \ref{cor:ket} below).  

Before proceeding, we will introduce some more general `Tate
kernels':

 For each $n\geq 1$, let $\Gamma_n=\gal{F_\infty}{F_n}$ and
let $G_n=\gal{F_n}{F}=\Gamma/\Gamma_n$. Let $\ed{i}{n}$ be the natural
map
\[
\left(F_n^\times/(F_n^\times)^{p^n}(i-1)\right)^{G_n}\to \K(i-1).
\]

We define
\[
\gtatk{F}{i}{n}:=(\ed{i}{n})^{-1}(\dvv{\NN(i-1)^\Gamma}).
\]

Thus if $F=F_n$ (i.e. if $\mu_{p^n}\subset F^\times$) then 
$
\gtatk{F}{i}{n}\subset F^\times/(F^\times)^{p^n}(i-1)
$
and there exists a subgroup $\ntatk{F}{i}{n}$ of $F^\times$ and
containing $(F^\times)^{p^n}$ with the property that 
$\gtatk{F}{i}{n}=(\ntatk{F}{i}{n}/(F^\times)^{p^n})(i-1)$. In
particular, $\ntatk{F}{i}{1}=\tatk{F}{i}$ for all $i\in\Z$. 

We will
prove below (see \ref{cor:ket}) that there is a natural isomorphism 
\[
\gtatk{F}{i}{n}\cong \ket{2i-1}{S}{F}/p^n
\]
when $i\geq 2$ and when either $F=F_n$ or $i\not\equiv 0\pmod{p}$.

Note that when $i=1$, the natural map $F^\times/(F^\times)^{p^n}\to
(F_n^\times/(F_n^\times)^{p^n})^{G_n}$ is an isomorphism. To see this,
one can use the Kummer isomorphism $F^\times/(F^\times)^{p^n}\cong
\coh{1}{F}{\mu_{p^n}}$ and the fact that the restriction map
$\coh{1}{F}{\mu_{p^n}}\to\coh{1}{F_n}{\mu_{p^n}}^{G_n}$ is an
isomorphism since $\coh{1}{G_n}{\mu_{p^n}}=\coh{2}{G_n}{\mu_{p^n}}=0$.
Thus, for all $n\geq 1$ 
$
\gtatk{F}{1}{n}\subset F^\times/(F^\times)^{p^n}
$
and 
if we let 
\[
\ntatk{F}{1}{n}:=\{ a\in \units{F}\ |a\otimes\frac{1}{p^n}\in \dvv{\NN^\Gamma}\}
\]
  then
\[
\gtatk{F}{1}{n}=\ntatk{F}{1}{n}/(F^\times)^{p^n}. 
\]

We let $p^e=p^{e_F}$ be the
exponent of the group $\h{F}$.
For all $n\geq 1$, we will let
$m_n=m_n(F)=\mathrm{max}(0,e_F+n-n_F)$. 
Observe that $m_n=0$ if there is an $e\geq 0$ with
 $p^e\h{F}=0$ and $\mu_{p^{e+n}}\subset F$. 

Observe that $m_n(F_m)=0$ whenever $m$ is suffiently large. Below we
will also let $m_F$ denote $m_1(F)$.

Let $\kappa:\Gamma\to 1+p\Z_p$ be the Teichm\"uller character of $F_\infty/F$ and
observe that, by definition of $m_n=m_n(F)$,
for all $\gamma\in \Gamma$, $\kappa(\gamma)^{p^{m_n}}\equiv
1\pmod{p^{e+n}}$.

The following theorem is essentially due to Greenberg (see
also \cite{am:cap}, Lemma 2.1):

\begin{thm} \label{thm:period}
 For any number field $F$ containing $\mu_p$, 
\[
\dvv{\NN(t)^\Gamma}[p^n]=\left(\dvv{\NN(t')^\Gamma}[p^n]\right)(t-t')
\mbox{ whenever } t\equiv t'\pmod{p^{m_n}}.
\]
\end{thm}
\begin{proof}
The argument we give here is just an adaptation of the proof of
\cite{am:cap}, Lemma 2.1, which treats the case $n=1$ and $m_1=0$. That argument
in turn is just a more explicit version of the sketch given by
Greenberg in \cite{gr:zp}. 

Observe that since $\Gamma$ acts trivially on
$\dvv{\NN(t)^\Gamma}[p^n]$, the assertion simply concerns the equality
of two subgroups of $\NN(t)$. 

As noted above,  $\NN(t+1)$ is Pontryagin-dual to $\XX(-t)$, so that
$\inv{\Gamma}{\NN(t+1)}$ is dual to
$\coinv{\Gamma}{\XX(-t)}=\XX(-t)/T\XX(-t)$, which is a
finitely-generated $\Z_p$-module. 

Hence
$\dvv{\NN(t+1)^\Gamma}$ is dual to 
$\free{\Z_p}{\XX(-t)/T\XX(-t)}$
and, finally,
$\dvv{\NN(t+1)^\Gamma}[p^n]$ is dual to $\free{\Z_p}{\XX(-t)/T\XX(-t)}/p^n$.
Therefore, we must prove that
\[
\free{\Z_p}{\XX(i)/T\XX(i)}/p^n=\free{\Z_p}{\XX(i)/T\XX(j)}/p^n(i-j)
\]
 as
quotient groups of $\XX(i)$ whenever $i\equiv j\pmod{p^{m_n}}$.

Now,  by the result of Iwasawa  mentioned above there is a  short exact sequence of $\Lambda$-modules
 
\[
0\to \XX(i)\to \Lambda(i)^{r_2}\to \h{F}(i)\to 0
\] 
for all $i\in\Z$.

From the commutative exact diagram
\begin{eqnarray*}
\xymatrix{
0\ar[r]
&
\XX(i)\ar[r]\ar[d]^T
&
\Lambda(i)^{r_2}\ar[r]\ar[d]^T
&
\h{F}(i)\ar[r]\ar[d]^T
&
0\\
0\ar[r]
&
\XX(i)\ar[r]
&
\Lambda(i)^{r_2}\ar[r]
&
\h{F}(i)\ar[r]
&
0\\
}
\end{eqnarray*}
we obtain the exact sequence
\begin{eqnarray*}
\xymatrix{
0\ar[r]
& \h{F}(i)[T]\ar[r]^{\delta}
&
\XX(i)/T\XX(i)\ar[r]
&
\Lambda(i)^{r_2}/T\Lambda(i)^{r_2}\ar[r]
&
\h{F}(i)/T\h{F}(i)\ar[r]
&
0.\\
}
\end{eqnarray*}
The image of $\delta$ is the group 
\[
\XX(i)\cap T\Lambda(i)^{r_2}/T\XX(i).
\]
But this is precisely the $\Z_p$-torsion subgroup of $\XX(i)/T\XX(i)$,
since $\Lambda(i)^{r_2}/T\Lambda(i)^{r_2}\cong
(\Lambda/T\Lambda)^{r_2}\cong \Z_p^{r_2}$. Thus 

\[
\free{\Z_p}{\XX(i)/T\XX(i)}/p^n=\XX(i)/Y_i\mbox{ where } Y_i=\XX(i)\cap T\Lambda(i)^{r_2}+p^n\XX(i).
\]

Therefore we must prove that whenever $i\equiv j\pmod{p^{m_n}}$, then
$Y_j(i-j)=Y_i$ as subgroups of $\XX(i)$. Note that
$p^{e+n}\Lambda(i)^{r_2}\subset Y_i$ since $p^e$ annihilates $\h{F}(i)$
and $\XX(i)/Y_i$ is annihilated by $p^n$.

 Fix a topological generator, $\gamma_0$, of
$\Gamma$, so that the action of $\Lambda=\Z_p[[T]]$ on a
$\Gamma$-module is given by $Tm:=(\gamma_0-1)m$.  

Suppose now that $y\in Y_j$ and that $i\equiv j\pmod{p^{m_n}}$. Let
$\zeta\in \Z_p(i-j)$. We must show that $y\otimes\zeta\in Y_i$. 
By definition of $Y_j$, $y=T\lambda+p^nx$ where
$\lambda\in\Lambda(j)^{r_2}$ and $T\lambda,x \in \XX(j)$. Let
$\zeta\in \Z_p(i-j)$. Thus 
\begin{eqnarray*}
(1+T)(\lambda\otimes \zeta) & = & \gamma_0(\lambda\otimes \zeta)\\
                            &=&
                            \kappa(\gamma_0)^{i-j}(\gamma_0\lambda\otimes
                            \zeta)\\
&=&\gamma_0\lambda\otimes\zeta+\left(\kappa(\gamma_0)^{i-j}-1\right)(\gamma_0\lambda\otimes
                            \zeta)\\
&=&\gamma_0\lambda\otimes\zeta+y'\\
&=&(1+T)\lambda\otimes\zeta+y'
\end{eqnarray*}
where $y'\in p^{e+n}\Lambda(i)^{r_2}\subset Y_i$. Thus
\begin{eqnarray*}
T(\lambda\otimes\zeta)=T\lambda\otimes\zeta+y'
\end{eqnarray*}
and hence
\begin{eqnarray*}
y\otimes\zeta&=& T\lambda\otimes\zeta+p^n(x\otimes\zeta)\\
&=& T(\lambda\otimes\zeta)+p^n(x\otimes\zeta)-y'
\end{eqnarray*}
which belongs to $Y_i$ since $y'\in Y_i$ and $y\otimes\zeta, y',
p^n(x\otimes\zeta)\in \XX(i) \imp T(\lambda\otimes \zeta)\in \XX(i)$. 

Thus, $Y_j(i-j)\subset Y_i$ and similarly $Y_i\subset Y_j(i-j)$,
proving the result.
\end{proof}

\begin{cor}\label{cor:period}
$\gtatk{F}{i}{n}(j-i)=\gtatk{F}{j}{n}$ whenever $i\equiv j\pmod{p^{m_n}}$.
\end{cor}
\begin{proof}
Note first that the conditions on $i$ and $j$ guarantee that $G_n$
acts trivially on $\mu_{p^n}^{\otimes(j-i)}$ since $p^{m_n}\geq
p^{n-n_F}=|G_n|$. The result now follows from the commutative diagram
\begin{eqnarray*}
\xymatrix{
\left(\frac{\units{F_n}}{(\units{F_n})^{p^n}}(i-1)\right)^{G_n}(j-i)
\ar[r]^{\ed{i}{n}(j-i)}\ar[d]^{=}
&
\dvv{\NN(i-1)^\Gamma}[p^n](j-i)\ar[d]^{=}\\
\left(\frac{\units{F_n}}{(\units{F_n})^{p^n}}(j-1)\right)^{G_n}\ar[r]^{\ed{j}{n}}
&
\dvv{\NN(j-1)^\Gamma}[p^n]
}
\end{eqnarray*}
\end{proof}
\begin{cor} If $p^e\h{F}=0$ and $\mu_{p^{e+n}}\subset F$, 
then $\gtatk{F}{i}{n}(j-i)=\gtatk{F}{j}{n}$ for all 
$i,j\in\Z$. 
\end{cor}

In section \ref{sec:coh} below we will see that if $p$ does not divide
$i$ or if $\mu_{p^n}\subset F$
there is a short exact sequence
\[
0\to \mu\to\gtatk{F}{i}{n}\to \dvv{\NN(i-1)^\Gamma}[p^n]\to 0.
\] 
where $\mu=\proots{F_i}^{\otimes i}/p^n$ is a nontrivial cyclic group
of order dividing $p^n$.

In particular, $\dim{\F{p}}{\tatk{F}{i}/(\units{F})^p}=1+r_2$ for all $i\in\Z$.

We now consider the structure of $\gtatk{F}{1}{n}$. 

For general number fields, the identification of $\gtatk{F}{1}{n}$ is
related to the Gross  conjecture (as extended by Jaulent). 

Let $U^S_F$ be the group of $S$-units of $F$. The homomorphism 
\[
G_F: U^S_F\to \coprod_{v|p}\Z_p\cdot v,\ \epsilon\mapsto
\sum_{v|p}\log_p[\norm{F_v}{\Q_p}{\epsilon}]\cdot v
\]
extends linearly to a $\Z_p$-module homomorphism 
\[
g_F: U^S_F\otimes \Z_p\to
\coprod_{v|p}\Z_p\cdot v.
\]
The image of $g_F$ is free $\Z_p$-module of rank at most
$|S_p(F)|-1$. Thus the rank of $\image{g_F}$ is $|S_p(F)|-1-\delta_F$
for some $\delta_F\geq 0$. Since the $\Z_p$-rank of $U^S_F\otimes
\Z_p$ is $r_2+|S_p(F)|-1$, it follows that $\ker{g_F}$ has rank
$r_2+\delta_F$. For more details see Sinnott, \cite{gross:sinnott} and
Kuzmin, \cite{kuz:tate}.

The conjecture of Gross (as extended by Jaulent) is:
\[
\delta_F=0.
\]

Jaulent has shown that this conjecture holds for all abelian fields.
Kolster shows that if $\delta_E=0$, then $\delta_F=0$ for all
subfields $F$ of $E$.

Following Jaulent, we will denote the Gross kernel $\ker{g_F}$ by $\gk{F}$. It is
also called the group of \emph{logarithmic units} of $F$.


Clearly, if the field $F$ has only one $p$-adic prime, then it is
trivially true that $\delta_F=0$ . Thus, in this case, $g_F$ is the zero
map and  $\gk{F}=U^S_F\otimes \Z_p$. 

In the next lemma and in the final section of this paper we  will need
the following result of Jaulent concerning logarithmic class groups:

We will let $\logcls{F}$ denote the logarithmic class group of
$F$. (For the logarithmic class group, see Jaulent, \cite{jaulent:sauv}
and \cite{jaulent:log}.) Thus there is an exact sequence 
\[
0\to \gk{F}\to \rk{F}\to \dl{F}\to \logcls{F}\to 0
\]
where $\rk{F}=\units{F}\otimes\Z_p$ and $\dl{F}$ is the \emph{group of
  logarithmic divisors of degree $0$}.

For a Galois extension $E/F$ with Galois group $G$, we let $f_1$ be
the natural functorial homomorphism $\logcls{F}\to (\logcls{E})^G$. 

\begin{thm}\label{thm:jaul}
Let $E/F$ be a finite layer of the cyclotomic $\Z_p$-extension of
$F$ and let $G=\gal{E}{F}$. Then 
\[
\ker{f_1}\cong  \coh{1}{G}{\gk{E}}\mbox{ and }\coker{f_1}\cong \coh{2}{G}{\gk{E}}
\]
If the Gross conjecture holds for $(F,p)$ these are finite groups of
the same order.
\end{thm} 
\begin{proof}
This follows immediately by applying Jaulent's genus formula for the logarithmic
class group (\cite{jaulent:log}, Theor\'eme 4.5) to the special case of a
$p$-cyclotomic extension, and using also Corollaire 3.7 (i) of the
same paper.
\end{proof}

The following lemma is an analogue for the Gross kernel of Iwasawa, \cite{iwa:zp}, Lemma 7
(which concerns the $S$-units). 

\begin{lem}\label{lem:gross}
Let $\gk{F_\infty}=\colim_n\gk{F_n}$. Then there is a short exact
sequence 
\[
0\to \gk{F}\otimes_{\Z_p}\Q_p/\Z_p\to
\left(\gk{F_\infty}\otimes_{\Z_p}\Q_p/\Z_p\right)^\Gamma\to
\coh{1}{\Gamma}{\gk{F_\infty}}\to 0
\]
where $\coh{1}{\Gamma}{\gk{F_\infty}}$ is finite.
\end{lem}
\begin{proof}
We begin by observing that $(\gk{F_n})^\Gamma=\gk{F}$ for all
$n$. This can be seen by taking $\Gamma$-invariants of the commutative
diagram with exact rows
\begin{eqnarray*}
\xymatrix{
0\ar[r]
&
\gk{F}\ar[r]\ar[d]
&
U^S_F\otimes \Z_p\ar[r]^{g_F}\ar[d]
&
\oplus_{v|p}\Z_p\ar[d]^{\Delta}
\\
0\ar[r]
&
\gk{F_n}\ar[r]
&
U^S_{F_n}\otimes \Z_p\ar[r]^{g_{F_n}}
&
\oplus_{v|p}\left(\oplus_{w|v}\Z_p\right)\\
}
\end{eqnarray*}
(where $\Delta$ is the map $(a_v)_v\mapsto (([F_{n,w}:F_v]a_v)_w)_v$)
and using the fact that
$(U^S_{F_n}\otimes\Z_p)^\Gamma=U^S_F\otimes\Z_p$ and that the vertical
arrows are all injective.

It follows, on taking limits, that $(\gk{F_\infty})^\Gamma=\gk{F}$.

Note that the torsion submodule of $\gk{F}$ is
$\proots{F}=\proots{F}\otimes\Z_p$. Let $\gkk{F}=\gk{F}/\proots{F}$.
Then $\gkk{F}$ is a free $\Z_p$-module and hence
$\gkk{F_\infty}=\colim_n(\gkk{F_n})$ is $\Z_p$-flat (in fact, it is
easy to verify that it is a \emph{free} $\Z_p$-module).

Now apply the exact functor $\gkk{F_\infty}\otimes_{\Z_p}$ to the
short exact sequence 
\[
0\to \Z_p\to \Q_p\to \Q_p/\Z_p\to 0
\]
to get the short exact sequence of $\Gamma$-modules
\[
0\to\gkk{F_\infty}\to \gkk{F_\infty}\otimes_{\Z_p}\Q_p\to
\gkk{F_\infty}\otimes_{\Z_p}\Q_p/\Z_p\to 0.
\]
Taking $\Gamma$-invariants gives the long exact sequence
\begin{eqnarray*}
0\to \gkk{F}\to \gkk{F}\otimes_{\Z_p} \Q_p\to
\left(\gkk{F_\infty}\otimes_{\Z_p}\Q_p/\Z_p\right)^\Gamma
\to \coh{1}{\Gamma}{\gkk{F_\infty}}\to \coh{1}{\Gamma}{\gkk{F_\infty}\otimes_{\Z_p}\Q_p}\cdots
\end{eqnarray*}

But  $\coh{1}{\Gamma}{\gkk{F_\infty}\otimes_{\Z_p}\Q_p}=0$ since
$(\gkk{F_\infty}\otimes_{\Z_p}\Q_p)^{\Gamma_n}=\gkk{F_n}\otimes_{\Z_p}\Q_p$
and $\coh{1}{G_n}{\gkk{F_n}\otimes_{\Z_p}\Q_p}=0$ for all $n$ (where
$\Gamma_n=\gal{F_\infty}{F_n}$ and $G_n=\Gamma/\Gamma_n$).

Now observing  that
$\gkk{F_n}\otimes_{\Z_p}\Q_p/\Z_p=\gk{F_n}\otimes_{\Z_p}\Q_p/\Z_p$ for
all $n$ and 
that $\coh{1}{\Gamma}{\gkk{F_\infty}}=\coh{1}{\Gamma}{\gk{F_\infty}}$ 
(since
$\coh{1}{\Gamma}{\mu_{p^\infty}}=\coh{2}{\Gamma}{\mu_{p^\infty}}=0$)
gives the short exact sequence we require.

Finally, we must show that $\coh{1}{\Gamma}{\gk{F_\infty}}$ is
finite. Of course,
$\coh{1}{\Gamma}{\gk{F_\infty}}
=\colim_n\coh{1}{G_n}{\gk{F_n}}$ since
$\gk{F_n}=(\gk{F_\infty})^\Gamma$. However, by Theorem \ref{thm:jaul}
\[
\coh{1}{G_n}{\gk{F_n}}=\ker{\logcls{F}\to\logcls{F_n}}.
\]
 Thus
$\coh{1}{\Gamma}{\gk{F_\infty}}\subset \logcls{F}$. 
Since $\coh{1}{\Gamma}{\gk{F_\infty}}$ is a torsion $\Z_p$-module and
$\logcls{F}$ is a finitely-generated $\Z_p$-module, the result follows. 
\end{proof} 

As an immediate corollary, we have:

\begin{cor}
\[
\dvv{\left(\gk{F_\infty}\otimes_{\Z_p}\Q_p/\Z_p\right)^\Gamma}=\gk{F}\otimes_{\Z_p}\Q_p/\Z_p.
\]
\end{cor}

\begin{thm}\label{thm:nn}
\[
\dvv{\inv{\Gamma}{\NN}}\subset \gk{F}\otimes_{\Z_p}\Q_p/\Z_p\mbox{ with
  equality if and only if }\delta_F=0.
\]
\end{thm}
\begin{proof}
M. Kolster has proven (\cite{kol:idel}, Theorem 2.3) that $\NN\subset
  \gk{F_\infty}\otimes_{\Z_p}\Q_p/\Z_p$ and thus taking
 $\Gamma$-invariants and then maximal divisible subgroups and using
 the last corollary gives the result. The statement about equality
  follows from the fact that $\gk{F}\otimes\Q_p/\Z_p\cong (\Q_p/\Z_p)^{r_2+\delta_F}$.
\end{proof}

Recall that $\ntatk{F}{1}{n}=\{ a\in \units{F}| a\otimes 1/p^n \in \dvv{\NN^\Gamma}\}$. 

\begin{cor} For any number field $F$ containing $\mu_p$,
  $\ntatk{F}{1}{n}\subset U^S_F\cdot(\units{F})^{p^n}$ with equality if and
  only if $F$ has exactly one $p$-adic prime. 
\end{cor}
\begin{proof}
Since $\gk{F}\otimes_{\Z_p}\Q_p/\Z_p\subset U^S_F\otimes\Q_p/\Z_p$
with equality if $F$ has one $p$-adic prime, it
follows that 
\[
\ntatk{F}{1}{n}\subset\{ a\in \units{F}|a\otimes\frac{1}{p^n} \in
U^S_F\otimes\Q_p/\Z_p\}=U^S_F\cdot(\units{F})^{p^n}.
\]
with equality when $F$ has one $p$-adic prime.
\end{proof}
\begin{rem}
In fact, Theorem 17 Iwasawa \cite{iwa:zp} implies that $\NN\subset
U^S_{F_\infty}\otimes \Q_p/\Z_p$ and Lemma 7 of that paper implies
that $\dvv{(U^S_{F_\infty}\otimes
  \Q_p/\Z_p)^\Gamma}=U^S_F\otimes\Q_p/\Z_p$, so that this corollary
can also be derived directly from the results of Iwasawa.
\end{rem}
\begin{cor}\label{cor:gk}
If the Gross conjecture holds for $(F,p)$, then 
there is a natural isomorphism
\[
\gtatk{F}{1}{n}\cong \gk{F}/p^n.
\]
\end{cor}
\begin{proof} Since $\delta_F=0$, we have 
\[
\dvv{\NN^\Gamma}[p^n]=(\gk{F}\otimes_{\Z_p}\Q_p/\Z_p)[p^n]=(\gkk{F}\otimes_{\Z_p}\Q_p/\Z_p)[p^n]=\gkk{F}\otimes_{\Z_p}
\left(\frac{1}{p^n}\Z_p/\Z_p\right)=\gkk{F}/p^n
\]
since $\gkk{F}$ is a free $\Z_p$-module.
Thus there is a natural short exact sequence 
\[
0\to \proots{F}/p^n\to \ntatk{F}{1}{n}/(\units{F})^{p^n}\to \gkk{F}/p^n\to 0.
\]
On the other hand, since $(\units{F})^{p^n}\subset \ntatk{F}{1}{n}\subset
U^S_F(\units{F})^{p^n}$, there is a natural isomorphism 
\[
\gtatk{F}{1}{n}=\ntatk{F}{1}{n}/(\units{F})^{p^n}\cong (U^S_F\cap\ntatk{F}{1}{n})/(U^S_F)^{p^n}.
\]
Now let $\epsilon\in U^S_F\cap \ntatk{F}{1}{n}$. Let
$\pun{F}=(U^S_F\otimes\Z_p)/\proots{F}$ so that $\gkk{F}\subset
\pun{F}$ as a $\Z_p$ direct summand. Then
\[
\epsilon\otimes \frac{1}{p^n}\in
\gkk{F}\otimes_{\Z_p}\left(\frac{1}{p^n}\Z_p/\Z_p\right) \subset \pun{F}\otimes_{\Z_p}\left(\frac{1}{p^n}\Z_p/\Z_p\right) 
\]
so that the image of $\epsilon\otimes 1$ in $\pun{F}/p^n$ lies in
$\gkk{F}/p^n$. It follows that $\epsilon\otimes 1\in\gk{F}$ since
$\proots{F}\otimes \Z_p\subset \gk{F}$. Thus we obtain a natural
well-defined homomorphism $\gtatk{F}{1}{n}\to \gk{F}/p^n$
which is an isomorphism in view of the commutative diagram with exact
rows
\begin{eqnarray*}
\xymatrix{
0\ar[r]
&
\proots{F}/p^n\ar[r]\ar[d]^{=}
&
\gtatk{F}{1}{n}\ar[r]\ar[d]
&
\gkk{F}/p^n\ar[d]^{=}\ar[r]
&
0
\\
0\ar[r]
&
\proots{F}/p^n\ar[r]
&
\gk{F}/p^n\ar[r]
&
\gkk{F}/p^n\ar[r]
&
0
\\
}
\end{eqnarray*}
\end{proof}

\begin{rem} This result (together with Corollary \ref{cor:period}) implies and clarifies the result of
  Brauckmann, \cite{brauck:thesis} (see also Kolster
  and Movahhedi, \cite{mk:codes}, Theorem 2.15 and Corollary 2.16)
  that $\tatk{F_m}{i}/(\units{F_m})^p\cong \gk{F_m}/p$ for all $i\geq
  2$ when $m$ is sufficiently large and assuming the Gross conjecture
  for $(F_m,p)$.
\end{rem}

\begin{cor}
Suppose that $F$ has exactly one $p$-adic prime. 

Then 
$\ntatk{F}{i}{n}=U^S_F(\units{F})^{p^n}$ for all $i\equiv 1\pmod{p^{m_n(F)}}$. 

If furthermore $m_F=0$, then $U^S_F(\units{F})^p\subset A_F$ and Leopoldt's conjecture holds for the field $F$ if and only if 
 there is equality.
\end{cor}

\begin{rem} \label{rem:h} Suppose that $F$ has only one prime dividing
 $p$. Then  $\h{F}=0$ if $\psyl{\cl{F}}{p}=0$, since under these
  hypotheses, $F_n/F$ is (totally) ramified at the unique $p$-adic
  prime and hence $\psyl{\cl{F_n}}{p}=0$ for all $n$.
 
It can be shown futhermore (see Greenberg, \cite{gr:zp}, p1241) 
that if $F$ is a CM-field with only one prime dividing
$p$ and if  $\psyl{\cl{F_+}}{p}=0$, then $\h{F}=0$. If Vandiver's
conjecture holds for the prime $p$, then this latter condition
holds for the cyclotomic field $F=\Q(\zeta_p)$ 

 \end{rem}

\begin{exa}
Thus, for example, we can deduce the following theorem of Kersten:
(\cite{ik:zp})

If $n$ is sufficiently large and $F=\Q(\zeta_{p^n})$, then
\[
A_F=U^S_F(\units{F})^p=C_F.
\]

( These fields have only one
$p$-adic prime, so that $\tatk{F}{1}=U^S_F\cdot(\units{F})^p$ for all $n$. They are abelian fields
and thus the Leopoldt  conjecture hold for
these fields and $\tatk{F}{0}=A_{F}$ for all $n$. Finally, $m_F=0$ for $n$ sufficiently large.) 
\end{exa}
\begin{exa} Greenberg (\cite{gr:zp}) proves that if $|\h{F}|=p$ and
  $\mu_{p^2}\not\subset \units{F}$ (so that $m=1$), then $\tatk{F}{i}=\tatk{F}{j}$ if
  \emph{and only if} $i\equiv j\pmod{p}$. 

When $p=3$, he shows that the field $F=\Q(\sqrt{257},\sqrt{-3})$
satisfies this condition.  There is a unique $3$-adic prime in this
field. Thus $\tatk{F}{1}=U^S_F(\units{F})^p$. Since $F$ is an abelian
number field, Leopoldt's Conjecture holds for $F$ and thus
$A_F=\tatk{F}{0}$. As observed above, $\tatk{F}{2}=C_F$, the classical
Tate kernel. Thus in this case we can conclude that the groups $A_F$,
$U^S_F(\units{F})^p$ and $C_F$ are pairwise distinct.
\end{exa}

\section{Cohomological Interpretation of the Generalised Tate Kernels}\label{sec:coh}

Let $E$ be any subfield  of $K$ which is Galois over $F$. (We are
primarily interested in the two cases $E=K$ and $E=\NF$.) Let
$\E\subset \K$ be
its dual with respect to the Kummer pairing. Let $G=\gal{E}{F}$. 
Let $Y=\gal{E}{F_\infty}$ so that there is a group extension
\begin{eqnarray}\label{eqn:seq}
1\to Y\to G\to \Gamma\to 1.
\end{eqnarray}

Thus
there is a perfect duality pairing $\E\times Y\to \mu_{p^\infty}$. Taking, the
$\Gamma$-module structure into account, this gives a perfect pairing
\[
\E(i-1)\times Y(-i) \to \Q_p/\Z_p
\]
for all $i \in \Z.$ 

\begin{lem} For all $i\in\Z$, there are natural isomorphisms of $\Gamma$-modules
\[
\coh{1}{Y}{\Q_p/\Z_p(i)}\cong \E(i-1).
\]
\end{lem}
\begin{proof} Since $Y$ acts trivially on $\mu_{p^n}\subset
  \units{F_\infty}$ for all $n$, we have
\begin{eqnarray*}
\coh{1}{Y}{\Q_p/\Z_p(i)}&=&\hom{}{Y}{\Q_p/Z_p(i)}\\
                        &\cong & \hom{}{Y(-i)}{\Q_p/\Z_p}\\
                        &\cong & \E(i-1).
\end{eqnarray*}
\end{proof}
\begin{lem} \label{lem:coh}
(a) For $i\not= 0$, there is a natural isomorphism
\[
\coh{1}{G}{\Q_p/\Z_p(i)}\cong \inv{\Gamma}{\E(i-1)}.
\]

(b) There is a (split) short exact sequence
\[
0\to \coh{1}{\Gamma}{\Q_p/\Z_p}\to \coh{1}{G}{\Q_p/\Z_p}\to
\inv{\Gamma}{\E(-1)}\to 0.
\]
\end{lem}
\begin{proof} (a) We begin with the observation (the `Lemma of Tate'),
  that when $i\not= 0$,
  $\coh{1}{\Gamma}{\Q_p/\Z_p(i)}=\coh{2}{\Gamma}{\Q_p/\Z_p(i)}=0$. 
Thus, the
  sequence of terms of low degree of the spectral sequence associated
  to the extension (\ref{eqn:seq}) yields a natural isomorphism
\[
\coh{1}{G}{\Q_p/\Z_p(i)}\cong\inv{\Gamma}{\coh{1}{Y}{\Q_p/\Z_p(i)}}\cong \inv{\Gamma}{\E(i-1)}.
\]

(b) When $i=0$, we have $\coh{1}{\Gamma}{\Q_p/\Z_p}\cong \Q_p/\Z_p$
and $\coh{2}{\Gamma}{\Q_p/\Z_p}=0$. Thus the sequence of terms of low
degrees gives the short exact sequence
\[
0\to \coh{1}{\Gamma}{\Q_p/\Z_p}\to \coh{1}{G}{\Q_p/\Z_p}\to
\coh{1}{Y}{\Q_p/\Z_p}^\Gamma\to 0
\] 
which is split since $\coh{1}{\Gamma}{\Q_p/\Z_p}$ is divisible.
\end{proof}

Consider now the sequences of coefficient modules
\begin{eqnarray}
\xymatrix{
0\ar[r]
&
\mu_{p^n}^{\otimes i}\ar[r]
&
\Q_p/\Z_p(i)\ar[r]^{p^n}
&
\Q_p/\Z_p(i)\ar[r]
&
0\\
}\label{seq:1}
\end{eqnarray}
and
\begin{eqnarray}
\xymatrix{
0\ar[r]
&
\Z_p(i)\ar[r]^{p^n}
&
\Z_p(i)\ar[r]
&
\mu_{p^n}^{\otimes i}\ar[r]
&
0.\\
}\label{seq:2}
\end{eqnarray}

Let $\alpha^n_G$ be the composite homomorphism
\begin{eqnarray*}\xymatrix{
\coh{1}{G}{\mu_{p^n}^{\otimes i}}\ar[r]
&
 \coh{1}{G}{\Q_p/\Z_p(i)}[p^n]\ar[r]
&
 \inv{\Gamma}{\E(i-1)}[p^n]\\
}
\end{eqnarray*}
associated to the sequence (\ref{seq:1}). Thus $\alpha^n_G$ is surjective.

Let $j_n$ be the injective homomorphism
\[
\coh{1}{G}{\Z_p(i)}/p^n\to \coh{1}{G}{\mu_{p^n}^{\otimes i}}
\]
associated to the sequence (\ref{seq:2}).

\begin{thm} \label{thm:coh}
The image of $j_n$ is
  $(\alpha^n_G)^{-1}(\dvv{\inv{\Gamma}{\E(i-1)}}[p^n])$ and there is a
  natural short exact sequence
\begin{eqnarray*}
\xymatrix{
0\ar[r]
&
 \proots{F_i}^{\otimes i}/p^n \ar[r]
&
\coh{1}{G}{\Z_p(i)}/p^n\ar[r]
&
 \dvv{\inv{\Gamma}{\E(i-1)}}[p^n]\ar[r]
&
0\\
}
\end{eqnarray*}
where we define $F_{-i}=F_i$ for $i>0$ and $\proots{F_0}^{\otimes 0}/p^n:=\Z/p^n$.
\end{thm}
\begin{proof} Bringing the coefficient sequence
\[
0\to \Z_p(i)\to \Q_p(i)\to \Q_p/\Z_p(i)\to 0
\] 
into play gives us the following commutative diagram with exact rows
and columns:
\begin{eqnarray*}
\xymatrix{
&
&
&
0\ar[d]
&
\\
&
0\ar[d]
&
&
\dvv{\coh{1}{G}{\Q_p/\Z_p(i)}}[p^n]\ar[d]
&\\
0\ar[r]
&
\coh{0}{G}{\Q_p/\Z_p(i)}/p^n\ar[r]\ar[d]
&
\coh{1}{G}{\mu_{p^n}^{\otimes i}}\ar[r]\ar[d]^{=}
&
\coh{1}{G}{\Q_p/\Z_p(i)}[p^n]\ar[d]\ar[r]
&
0\\
0\ar[r]
&
\coh{1}{G}{\Z_p(i)}/p^n\ar[r]^{j_n}
&
\coh{1}{G}{\mu_{p^n}^{\otimes i}}\ar[r]
&
\coh{2}{G}{\Z_p(i)}[p^n]\ar[r]\ar[d]
&
0\\
&
&
&
0
&
\\
}
\end{eqnarray*}

The statement about the image of $j_n$ follows by diagram-chasing.

For the second statement, observe first that for $i\not= 0$, 
$\coh{0}{G}{\Q_p/\Z_p(i)}=\proots{F_i}^{\otimes i}$. 

When $i=0$, 
$\coh{0}{G}{\Q_p/\Z_p}/p^n=0$ and 
$\coh{1}{G}{\Z/p^n}\cong \coh{1}{G}{\Q_p/\Z_p}[p^n]$. However, taking the
  maximal divisible subgroups and then $p^n$-torsion subgroups of the split-exact sequence of  Lemma
  \ref{lem:coh} (b), now gives a short exact sequence
\[
0\to\Z/p^n\to \coh{1}{G}{\Z/p^n}\to \inv{\Gamma}{\E(-1)}[p^n]\to 0.
\]    
\end{proof}

\begin{lem}\label{lem:Fn}
If either $\mu_{p^n}\subset F$ or $i\not\equiv 0\pmod{p}$, the
restriction map induces an isomorphism 
\[
\coh{1}{F}{\mu_{p^n}^{\otimes i}}\cong\coh{1}{F_n}{\mu_{p^n}^{\otimes i}}^{G_n} 
\]
\end{lem}
\begin{proof}
If $\mu_{p^n}\subset F$, then $F=F_n$ and $G_n=0$ and the statement is
trivial. Otherwise $G_n\not= 0$.

Let $G_F$ denote the absolute Galois group of $F$. The sequence of
terms of low degree associated to the extension 
\[
0\to G_{F_n}\to G_F\to G_n\to 0
\]
has the form
\[
0\to \coh{1}{G_n}{\mu_{p^n}^{\otimes i}}\to
\coh{1}{F}{\mu_{p^n}^{\otimes i}}\to\coh{1}{F_n}{\mu_{p^n}^{\otimes
    i}}^{G_n}\to
\coh{2}{G_n}{\mu_{p^n}^{\otimes i}}\ldots
\]

Because $i\not\equiv 0\pmod{p}$ the map $G_n\to
\aut{\mu_{p^n}^{\otimes i}}=(\Z/p^n)^\times$ is \emph{injective}. Thus
the proof concludes with the observation that if $M$ is cyclic of order
$p^n$ and if $H$ is a nonzero
subgroup of $(\Z/p^n)^\times=\aut{M}$ then
$\coh{1}{H}{M}=\coh{2}{H}{M}=0$.

(On the other hand, if $p|i$ then the map $G_n\to
\aut{\mu_{p^n}^{\otimes i}}$ has a nonzero kernel and it is then easy
to see that $\coh{1}{G_n}{\mu_{p^n}^{\otimes i}}\not= 0$.) 
\end{proof}

Recall that Kummer Theory gives a natural isomorphism
$\delta:\units{F}/(\units{F})^{p^n}\cong \coh{1}{F}{\mu_{p^n}}$
and hence for all $i\in \Z$ there are natural isomorphisms
$\coh{1}{F_n}{\mu_{p^n}^{\otimes i}}\cong \coh{1}{F_n}{\mu_{p^n}}\otimes
\mu_{p^n}^{\otimes (i-1)}\cong
\units{F_n}/(\units{F_n})^{p^n}(i-1)$.

\begin{cor} 
Let $\LL=\gal{\NF}{F}$. If either $p\not| i$ or $\mu_{p^n}\subset F$ 
 the image of the natural
injective map 
\[
\coh{1}{\LL}{\Z_p(i)}/p^n\to \coh{1}{\LL}{\mu_{p^n}^{\otimes i}}\to
\coh{1}{F}{\mu_{p^n}^{\otimes i}}\to (\units{F_n}/(\units{F_n})^{p^n}(i-1) 
\] 
is $\gtatk{F}{i}{n}$.
\end{cor}

\begin{proof}
Let $H=
\gal{K}{F}$. 
We begin with the observation that the natural map induces an
isomorphism $\coh{1}{H}{\mu_{p^n}^{\otimes
    i}}\cong\coh{1}{F}{\mu_{p^n}^{\otimes
    i}}=\coh{1}{G_F}{\mu_{p^n}^{\otimes i}}$. To see this, first note
that $\coh{1}{\mathcal{F}}{\mu_{p^n}^{\otimes
    i}}=\coh{1}{\mathcal{F_n}}{\mu_{p^n}^{\otimes i}}^{G_n}$ and
$\coh{1}{\mathcal{H}}{\mu_{p^n}^{\otimes
    i}}=\coh{1}{\mathcal{H_n}}{\mu_{p^n}^{\otimes i}}^{G_n}$ (where
$H_n=\gal{K}{F_n}$) since $\coh{1}{G_n}{\mu_{p^n}^{\otimes
    i}}=\coh{1}{G_n}{\mu_{p^n}^{\otimes i}}=0$. But we have
$\ab{H_n}/p^n=\ab{G_{F_n}}/p^n$ and thus 
$\coh{1}{H_n}{\mu_{p^n}^{\otimes i}}=\hom{}{H_n}{\mu_{p^n}^{\otimes
    i}}=
\hom{\ab{H_n}/p^n}{\mu_{p^n}^{\otimes
    i}}=\hom{}{\ab{G_{F_n}}/p^n}{\mu_{p^n}^{\otimes
    i}}=\coh{1}{F_n}{\mu_{p^n}^{\otimes i}}$.

The result then follows from the commutative diagram
\begin{eqnarray*}
\xymatrix{
\coh{1}{\LL}{\Z_p(i)}/{p^n}\ar[r]^{\alpha^n_{\LL}\circ j_n}\ar[d]^{j_n}
&
\dvv{\NN(i-1)^\Gamma}[p^n]\ar[d]\\
\coh{1}{\LL}{\mu_{p^n}^{\otimes i}}\ar[r]^{\alpha^n_{\LL}}\ar[d]
&
\NN(i-1)\ar[d]\\
\coh{1}{H}{\mu_{p^n}^{\otimes i}}\ar[r]^{\alpha^n_H}
&
\K(i-1)\\
\coh{1}{F}{\mu_{p^n}^{\otimes i}}\ar[u]_{\cong}
&
\left(\units{F_n}/(\units{F_n})^{p^n}(i-1)\right)^{G_n}\ar[l]_{\cong}\ar[u]_{\ed{i}{n}}
}
\end{eqnarray*}
together with Theorem \ref{thm:coh}.
\end{proof}

\begin{cor} If either $i\not\equiv 0\pmod{p}$ or $\mu_{p^n}\subset F$, there is a natural short exact sequence
\[
0\to \proots{F_i}^{\otimes i}/p^n\to \gtatk{F}{i}{n}\to
\dvv{\inv{\Gamma}{\NN(i-1)}}[p^n]\to 0.
\]
\end{cor}
\begin{lem} Suppose that
  $\dvv{\inv{\Gamma}{\K(i-1)}}=\dvv{\inv{\Gamma}{\NN(i-1)}}$
  (i.e. suppose that Schneider's conjecture holds for $i$). Let
  $\tilde{E}/F$ be a  Galois extension
  containing $\NF$ and  let $\tilde{G}=\gal{\tilde{E}}{F}$. Suppose
  that either $\mu_{p^n}\subset F$ or $p\not| i$. Then the
  natural restriction map induces an isomorphism
\[
\coh{1}{\LL}{\Z_p(i)}/p^n\cong \coh{1}{\tilde{G}}{\Z_p(i)}/p^n.
\]  
\end{lem}
\begin{proof}
Let $E/F$ be the largest sub-extension of $K/F$ which is contained in
$\tilde{E}/F$. Let $G=\gal{E}{F}$. I \emph{claim} that the restriction
map induces an isomorphism
\[
\coh{1}{G}{\Z_p(i)}/p^n\cong \coh{1}{\tilde{G}}{\Z_p(i)}/p^n.
\]
To see this observe that in the commmutative exact diagram
\begin{eqnarray*}
\xymatrix{
0\ar[r]
&
\coh{1}{G}{\Z_p(i)}/p^n\ar[r]\ar[d]^{\mathrm{res}}
&
\coh{1}{G}{\mu_{p^n}^{\otimes i}}
\ar[r]\ar[d]^{\rho}
&
\coh{1}{G}{\Q_p/\Z_p(i)}\ar[d]^{\tau}\\
0\ar[r]
&
\coh{1}{\tilde{G}}{\Z_p(i)}/p^n\ar[r]
&
\coh{1}{\tilde{G}}{\mu_{p^n}^{\otimes i}}
\ar[r]
&
\coh{1}{\tilde{G}}{\Q_p/\Z_p(i)}\\
}
\end{eqnarray*}
$\tau$ is injective and it is therefore sufficient to show that $\rho$
is an isomorphism. Let $K_0=\ker{\tilde{G}\to G}$. Then $\rho$ fits
into the exact sequence
\begin{eqnarray*}
\xymatrix{
0\ar[r]
&
\coh{1}{G}{\mu_{p^n}^{\otimes i}}\ar[r]^{\rho}
&
\coh{1}{\tilde{G}}{\mu_{p^n}^{\otimes i}}\ar[r]^{\psi}
&
\coh{1}{K_0}{\mu_{p^n}^{\otimes i}}\\
}
\end{eqnarray*}
and it suffices to show that $\psi$ is the zero map. Let
$G_0=\ker{\tilde{G}\to \Gamma}$. Then $K_0\subset G_0$ and $G_0/K_0$
is the largest abelian pro-$p$ quotient of $G_0$. The map $\psi$
factors through $\coh{1}{G_0}{ \mu_{p^n}^{\otimes
    i}}=\hom{}{\ab{G_0}/p^n}{\mu_{p^n}^{\otimes i}}\to 
\hom{}{K_0}{\mu_{p^n}^{\otimes i}}=\coh{1}{K_0}{\mu_{p^n}^{\otimes i}}$
which is the zero map since $K_0\subset\ker{G_0\to
  \ab{G_0}/p^n}$. This proves the \emph{claim}.

 Let now $\E$ be the dual
of $Y=\gal{E}{F_\infty}$. Thus $\NN\subset\E\subset\K$ and hence  
$\dvv{\inv{\Gamma}{\K(i-1)}}=\dvv{\inv{\Gamma}{\E(i-1)}}=\dvv{\inv{\Gamma}{\NN(i-1)}}$
so that we get the commutative diagram with exact rows 
\begin{eqnarray*}
\xymatrix{
0\ar[r]
&
\proots{F_i}^{\otimes i}/p^n\ar[r]\ar[d]^{=}
&
\coh{1}{\LL}{\Z_p(i)}/p^n\ar[r]\ar[d]^{\mathrm{res}}
&
\dvv{\inv{\Gamma}{\NN(i-1)}}[p^n]\ar[r]\ar[d]^{\cong}
&
0\\
0\ar[r]
&
\proots{F_i}^{\otimes i}/p^n\ar[r]
&
\coh{1}{G}{\Z_p(i)}/p^n\ar[r]
&
\dvv{\inv{\Gamma}{\E(i-1)}}[p^n]\ar[r]
&
0\\
}
\end{eqnarray*}
\end{proof}

Now let $G^S_F$ denote the Galois group of the maximal algebraic
extension of $F$ which is unramified outside $S$. 
\begin{cor}\label{cor:ket}
Let $i\geq 2$ and suppose that either $\mu_{p^n}\subset F$ or $p\not|
i$.  There are natural isomorphisms 
\[
\gtatk{F}{i}{n}\cong \ket{2i-1}{S}{F}/p^n.
\]
\end{cor}

\begin{proof} As remarked above,
  $\dvv{\inv{\Gamma}{\NN(i-1)}}=\dvv{\inv{\Gamma}{\K(i-1)}}$ for all
  $i\geq 2$. Furthermore, the \'etale $K$-theory groups
  $\ket{2i-1}{S}{F}$ are isomorphic to the groups
  $\coh{1}{G^S_F}{\Z_p(i)}$ (Dwyer and Friedlander, \cite{df:etale},
  Proposition 5.1).
\end{proof}

Compare this last result with the identification of
$\gtatk{F}{1}{n}$ in Corollary \ref{cor:gk} above. 

\section{Capitulation Kernels: General Results}
\label{sec:gen}

Let $E/F$ be  a  Galois extension of number fields, with Galois
group $G$. As in \cite{am:cap}, $f_i$ ($i\geq 2$) 
denotes the natural functorial homomorphism 
$\ket{2i-2}{S}{F}\to\ket{2i-2}{S}{E}^G$.  

From \cite{am:cap}, Propositions 1.1 (which is based on the work of
B. Kahn, \cite{kah:des}) and  the remarks that follow it,  we
have:
\begin{thm}\label{thm:am} Let $E/F$ be cyclic of degree $p^n$. Then
 $|\ker{f_i}|=|\coker{f_i}|$ and  
\[
\coker{f_i}\cong
(\ket{2i-1}{S}{F}/p^n)/\norm{E}{F}(\ket{2i-1}{S}{E}/p^n)  \mbox{ for all }
i\geq 2.
\] 
\end{thm}

\begin{cor}   Let $E/F$ be cyclic of degree $p^n$.  
Suppose that $\mu_{p^n}\subset F$ or that $p$ does not divide $i$.  Then
\[
|\ker{f_i}|=|\coker{f_i}|=[\gtatk{F}{n}{i}:\norm{E}{F}(\gtatk{E}{n}{i})]
\]
\end{cor}

\begin{proof} We have a commutative diagram 
\begin{eqnarray*}
\xymatrix{
\ket{2i-1}{S}{E}/p^n\ar[r]^{\cong}\ar[d]^{\norm{E}{F}}
&
\gtatk{E}{i}{n}\ar[d]^{\norm{E}{F}}\\
\ket{2i-1}{S}{F}/p^n\ar[r]^{\cong}
&
\gtatk{F}{i}{n}\\
}
\end{eqnarray*}
\end{proof}

In particular, we get the following result of Assim and Movahhedi:
\begin{cor}
Let $E/F$ be cyclic of degree $p$. Then for all $i\geq 2$, 
\[
\coker{f_i}\cong \tatk{F}{i}/(\norm{E}{F}{\tatk{E}{i}}(\units{F})^p).
\]
\end{cor}
\begin{proof}
When $n=1$, $\gtatk{F}{i}{1}=\tatk{F}{i}/(\units{F})^p$.
\end{proof}
\begin{cor}  \label{cor:f1} Let $E/F$ be cyclic of degree $p^n$.  
Suppose that the Gross conjecture holds for $(E,p)$. Let
  $m_n=\mathrm{max}(m_n(F),m_n(E))$. Then
\[
|\ker{f_i}|=|\coker{f_i}|=[\gk{F}:\norm{E}{F}(\gk{E})]\mbox{ for all
 }i\equiv 1\pmod{p^m}.
\]

If $E$ has only one $p$-adic prime, then
\[
|\ker{f_i}|=|\coker{f_i}|=[U^S_F:\norm{E}{F}(U^S_F)]\mbox{ for all
 }i\equiv 1\pmod{p^m}.
\]
\end{cor}

\begin{proof} By Corollary \ref{cor:gk} we have a commutative diagram 
\begin{eqnarray*}
\xymatrix{
\ket{2i-1}{S}{E}/p^n\ar[r]^{\cong}\ar[d]^{\norm{E}{F}}
&
\gtatk{E}{i}{n}\ar[r]^{\cong}\ar[d]^{\norm{E}{F}}
&
\gtatk{E}{1}{n}(i-1)\ar[r]^{\cong}\ar[d]^{\norm{E}{F}(i-1)}
&
\gk{E}/p^n(i-1)\ar[d]^{\norm{E}{F}(i-1)}\\
\ket{2i-1}{S}{F}/p^n\ar[r]^{\cong}
&
\gtatk{F}{i}{n}\ar[r]^{\cong}
&
\gtatk{F}{1}{n}(i-1)\ar[r]^{\cong}
&
\gk{F}/p^n(i-1)\\
}
\end{eqnarray*}
and the result follows since $p^n\gk{F}\subset \norm{E}{F}(\gk{E})$ and
$\mu_{p^n}^{\otimes (i-1)}$ is a trivial Galois-module.
\end{proof}

\section{Capitulation Kernels: Tamely Ramified $p$-extensions}
\label{sec:tame}

For a number field $F$ (containing $\mu_p$) we will set 
\[
B_F:=\{ a\in \units{F} | F(\sqrt[p]{a})/F\mbox{ is unramified outside }p\}.
\]

Thus $B_F$ is a subgroup of $\units{F}$ containing $U^S_F\cdot
(\units{F})^p$ and $A_F$. It can also be shown that $C_F\subset
B_F$. More generally, we have: 

\begin{lem}
$\tatk{F}{i}\subset B_F$ for all $i\in \Z$.
\end{lem}
\begin{proof}
Let $a\in \tatk{F}{i}$. Let $\zeta_p$ be a $p$th root of unity in
$F$. Then
$
a\otimes\zeta_p^{\otimes (i-1)}\in \dvv{\inv{\Gamma}{\NN(i-1)}}
\subset \M(i-1)$. Thus 
\[
a\otimes \frac{1}{p}\in \M 
\] 
and hence $\sqrt[p]{a}\in M$. It follows that $F(\sqrt[p]{a})$ is
unramified outside $p$ since $M/F$ is ramified only at $p$-adic primes.
\end{proof}

Suppose now that $W$ is any subgroup of $B_F$ containing
$(\units{F})^p$. Let $F_W=F(\sqrt[p]{W})$ and
$G_W=\gal{F_W}{F}$. Thus, Kummer theory gives a perfect pairing of
$\F{p}$-vectorspaces:
\begin{eqnarray*}
G_W\times W/(\units{F})^p & \to & \mu_p\\
(\sigma, w) & \mapsto & \sigma(\sqrt[p]{w})/\sqrt[p]{w}.
\end{eqnarray*}

Let $E/F$ be a cyclic degree $p$ extension. Let $\ram{E}{F}$ be the
set of primes of $F$ which ramify in this extension.

Let $H_W=H_W(E/F)$ be the subspace of $G_W$ spanned by the set
\[
\left\{ \sigma_v| v\in \ram{E}{F}\setminus S_p(F) \right\}
\] 
(where $\sigma_v$ denotes the Frobenius of $v$ in $F_W/F$).

Let $t_W=\dim{\F{p}}{H_W}$. 

\begin{thm}\label{thm:main}
Suppose that $E/F$ is a cyclic extension of degree $p$ and that $W$ is
a subgroup of $B_F$ containing $(\units{F})^p$. Then 
\[
[W:W\cap \norm{E}{F}(\units{E})]\geq p^{t_W}.
\]
with equality if $F$ has exactly one $p$-adic prime.
\end{thm}
\begin{proof}
We follow closely the argument of Assim and Movahhedi, (\cite{am:cap})
which deals with the case $W=A_F$.

In fact we will show that the complement, $H_W^\perp$,  of $H_W$ with respect to the
Kummer pairing  contains $(W\cap
\norm{E}{F}(\units{E}))/(\units{F})^p$ (and hence the dual of $H_W$ is a quotient
of  $W/W\cap\norm{E}{F}(\units{E})$).
This is equivalent to the statement
\[
\sigma(\sqrt[p]{w})=
  \sqrt[p]{w}\mbox{ for all } \sigma\in H_W \ \imp\ w\in \norm{E}{F}(\units{E})  
\]
which, in turn, is equivalent to the statement
\[
\exists v\in
\ram{E}{F}\setminus S_p(F)\mbox{ with }\sigma_v(\sqrt[p]{w})\not=
  \sqrt[p]{w} \ \imp\ w\not\in \norm{E}{F}(\units{E})\ 
\]
and hence to the statement
\[
\exists v\in
\ram{E}{F}\setminus S_p(F)\mbox{ with }
v\mbox{ inert  in }F(\sqrt[p]{w})/F \ \imp\ w\not\in \norm{E}{F}(\units{E}).
\]

Suppose, then, that there is a non-$p$-adic prime $v$ which is inert
in $K=F(\sqrt[p]{w})$ and ramified in $E/F$.
By Kummer Theory, $E=F(\sqrt[p]{b})$ for some $b\in F^\times$.
Then $K_{v'}/F_v$ is an unramified cyclic extension and hence
$\norm{K_{v'}}{F_v}(\units{K_{v'}})=U_{F_v}\cdot (\units{F_v})^p$.
But $b\not\in U_{F_v}\cdot (\units{F_v})^p$ since
$F_v(\sqrt[p]{b})/F_v=E_u/F_v$ is ramified. It follows that the Hilbert symbol
$
\artin{w}{b}{v}_p
$
is nontrivial, and hence $w\not\in \norm{E_u}{F_v}(\units{E_u})$.

Conversely, suppose that $|S_p(F)|=1$ and  that $w\not\in \norm{E}{F}(\units{E})$.
 We must show that there is a non-$p$-adic prime ramifying in $E$ but
 inert in $F(\sqrt[p]{w})=K$.

By
Hasse's norm theorem, there is a finite prime $v_0$ such that
$w\not\in\norm{E_{w_0}}{F_{v_0}}(E_{w_0}^\times)$. Hence
\[
\artin{w}{b}{v_0}_p\not= 1.
\]

By Artin's reciprocity law (and the fact that $|S_p|=1$) there exists
$v\not\in S_p$ with 
\[
\artin{w}{b}{v}_p\not= 1.
\]
Thus $w\not\in\norm{E_{u}}{F_{v}}(E_{u}^\times)$. In particular, $v$
does not split in $E$. 

Let $K=F(\sqrt[p]{w})$. Then $v$ does not split in $K$ either. So
$K_{v'}/F_v$ is a cyclic degree $p$ extension. But since $w$ is a norm
from $K_{v'}$, we have $K_{v'}\not= E_u$. But $K_v'/F_v$ is an
unramified extension since $v\not\in S_p(F)$ and $w\in B_F$, and hence is the \emph{unique}
cyclic degree $p$ unramified extension of $F_v$. It follows that
$E_u/F_v$ is ramified; i.e., $v$ ramifies in $E$ but is inert in $K$
as required.

\end{proof}

\begin{cor} Suppose that $E/F$ is a cyclic degree $p$ extension. Let
  $m=\mathrm{max}(m_F,m_E)$. Fix $j\in \Z$ and let $t_j:=
  t_{\tatk{F}{j}}(E/F)$. Then  
\[
|\ker{f_i}|\geq p^{t_j}\mbox{ for all }i\equiv j\pmod{p^m}\ (i\geq 2).
\]
\end{cor}
\begin{proof}
We have $\tatk{F}{i}=\tatk{F}{j}$ for all $i$ with $i\equiv
j\pmod{p^m}$. Thus,
for $i\geq 2$ and $i\equiv j\pmod{p^m}$ we thus have
\begin{eqnarray*}
|\ker{f_i}|&=&[\tatk{F}{i}:\norm{E}{F}(\tatk{E}{i})(\units{F})^p]\\
           &=&[\tatk{F}{j}:\norm{E}{F}(\tatk{E}{j})(\units{F})^p]\\           
           &\geq & [\tatk{F}{j}:\tatk{F}{j}\cap \norm{E}{F}{\units{E}}]\\
&\geq & p^{t_j}.
\end{eqnarray*}
\end{proof}
\begin{cor}
Suppose that $E/F$ is a
cyclic degree $p$ extension and that $F$ has exactly one $p$-adic prime. Let $m=\mathrm{max}(m_F,m_E)$.  Then 
\[
|\ker{f_i}|\geq p^t\mbox{ for all }i\equiv 1\pmod{p^m}\ (i\geq 2)
\]
where $t=t_{U^S_F}(E/F)$. 
\end{cor}

\begin{rem}
Observe that given a subgroup, $W$, of $B_F$, containing
$(\units{F})^p$, $t_W=t_W(E/F)$ is the maximal size of a set
$\{ v_1,\ldots,v_t\}$ of non-$p$-adic primes ramifying in $E$ and such
that $\sigma_{v_1},\ldots,\sigma_{v_t}$ is linearly independent in
$\gal{F(\sqrt[p]{W})}{F}$. In the case $W=A_F$,
$F(\sqrt[p]{W})=\tilde{F_1}$, the compositum of the first layers of
the $\Z_p$-extensions of $F$, and the set $\{ v_1,\ldots,v_t\}\cup
S_p$ is said to be \emph{primitive for $(F,p)$} (see \cite{gras-jaul}). 

For other subgroups
$W$ of $B_F$ (eg, $W=U^S_F$, $W=C_F$, $W=\tatk{F}{i}$ any $i$, etc) we can use the term
\emph{$W$-primitive for $(F,p)$}. Thus the number $t$ in the last corollary is the
maximal number of tamely-ramified primes in $E/F$ which belong to a
$U^S_F$-primitive set for $(F,p)$.  
\end{rem}

\begin{cor} Suppose that
  Leopoldt's conjecture holds for $(F,p)$. Suppose that $E/F$ is
  cyclic of degree $p$. Let $m=\mathrm{max}(m_F,m_E)$. Then 
\[
|\ker{f_i}|\geq p^t\mbox{ for all }i\equiv 0\pmod{p^m} \ (i\geq 2)
\]
where $t$ denotes the maximal number of tamely-ramified primes in
$E/F$ belonging to a primitive set for $(F,p)$. 
\end{cor}
\begin{proof}
Since Leopoldt's conjecture holds, we have $\tatk{F}{0}=A_F$.
\end{proof}
\section{Capitulation Kernels: Wildly Ramified Extensions}
\label{sec:wild}

The results of the last section give no information about the
situation in which there are no tamely-ramified primes. The lower bounds
obtained depend on the index $[\tatk{F}{i}:\tatk{F}{i}\cap
  \norm{E}{F}(\units{E})]$. 
However, when there is no tame ramification 
we have:

\begin{lem} 
Suppose that the field $F$ has exactly one $p$-adic prime and that $E/F$ is a cyclic degree $p$ extension in which all
non-$p$-adic primes are unramified. Then $B_F\subset
\norm{E}{F}(\units{E})$. 
\end{lem}
\begin{proof}
Taking $W=B_F$ in Theorem \ref{thm:main}. Then $t_W=0$ since
$\ram{E}{F}\setminus S_p(F)=\emptyset$. Thus $[B_F:B_F\cap\norm{E}{F}(\units{E})]=1$.
\end{proof}
\begin{cor}
Suppose that the field $F$ has exactly one $p$-adic prime and that $E/F$ is a cyclic degree $p$ extension in which all
non-$p$-adic primes are unramified. Then $\tatk{F}{i}\subset
\norm{E}{F}(\units{E})$ for all $i\in \Z$.
\end{cor}

Vandiver's conjecture implies that the groups $\ket{2i-2}{S}{F}$ satisfy
Galois descent in the $p$-cyclotomic tower of $\Q$:
\begin{lem}
Let $p$ be a prime number for which Vandiver's conjecture is true. Let
$n\geq 1$ and let $F=\Q(\zeta_{p^n})$ and
$E=\Q(\zeta_{p^{n+1}})$. Then 
\[
\ker{f_i}=\coker{f_i}=1\mbox{ for all } i\geq 2.
\]
\end{lem}
\begin{proof}
The assumption (Vandiver's conjecture) implies that $p$ does not
divide $|\cls{S}{F_+}|$ or $|\cls{S}{E_+}|$. It follows that
$\h{F}=\h{E}=0$ (see Remark \ref{rem:h}). 
Thus $m=\mathrm{max}(m_F,m_E)=0$. 

However, the fact that $p$ does not divide the
$S$-class number of $E_+$ or $F_+$ implies that
$\un{S}{F}/(\un{S}{F})^p$ and $\un{S}{E}/(\un{S}{E})^p$ are generated
by roots of unity and units of the form $1-\zeta^a$ where $\zeta$ is a
root of unity (this follows from the fact that
$[\un{}{F_+}:C(F_+)]=|\cls{S}{F_+}|$ where $C(F_+)$ is the subgroup of
cyclotomic units; see Washington, \cite{wash:cf},
Theorem 8.2). 

Since $\norm{E}{F}(\zeta_{p^{n+1}}^a)=\zeta_{p^n}^a$ and
$\norm{E}{F}(1-\zeta_{p^{n+1}}^a)=1-\zeta_{p^n}^a$, it follows that 
$\norm{E}{F}(\un{S}{E})=\un{S}{F}$.  
\end{proof}
On the other hand, Greenberg's results easily provide examples of
$p$-cyclotomic extensions for which the kernel and cokernel of the
maps $f_i$ are nontrivial for all $i$:
\begin{exa}
Suppose that $|\h{F}|=p$ and $\mu_{p^2}\not\subset \units{F}$. Thus
$m_F=1$ and, by Greenberg's work,  $\tatk{F}{i}=\tatk{F}{j}$ if and only if $i\equiv
j\pmod{p}$. Note that it follows, of course, that
$\tatk{F}{i}\not\subset\tatk{F}{j}$ whenever $i\not\equiv j\pmod{p}$,
since $(\units{F})^p\subset \tatk{F}{i}$ and
$\dim{\F{p}}{\tatk{F}{i}/(\units{F})^p}=1+r_2$ for all $i$.

Now let $E=F(\mu_{p^2})$. Then $\h{E}=\h{F}\imp |\h{E}|=p$, but
$\mu_{p^2}\subset \units{E}$. Thus, $m_E=0$ and so
$\tatk{E}{i}=\tatk{E}{j}$ for all $i,j\in\Z$. 

It follows that, for any $i$,
  $\norm{E}{F}(\tatk{E}{i})=\norm{E}{F}(\tatk{E}{j})\subset \tatk{F}{j}$ for all $j$ and
hence 
\[
\norm{E}{F}(\tatk{E}{i})\subset \cap_{j=0}^{p-1}\tatk{F}{j}\not=\tatk{F}{i}
\]  
for all $i$. Thus, for all $i\geq 2$ we have 
\[
\left\vert \ker{f_i}\right\vert = \left\vert
\coker{f_i}\right\vert\geq
      [\tatk{F}{i}:\cap_{j=0}^{p-1}\tatk{F}{j}]\geq p.
\]
\end{exa}

Let $E/F$ be a cyclic extension of degree $p$ with Galois group $G$.
Let $H(E/F)=H_{U^S_F}(E/F)$ be the subspace of
$G_{U^S_F}=\gal{F(\sqrt[p]{U^S_F})}{F}$ spanned by the Frobenius
automorphisms of those non $p$-adic primes which ramify in $E/F$. Let
$\dual{H(E/F)}$ be the dual vectorspace.

Let 
$f^S_1$ denote the natural functorial homomorphism
$\as{F}\to\Inv{G}{\as{E}}$.  
\begin{thm}\label{thm:coker}
 Suppose that $E$ has  exactly one $p$-adic prime.
 Let
 $m=\mathrm{max}(m_F,m_E)$. Let $t$ be the number of non $p$-adic
 primes of $F$ which ramify in $E$.
  
Then for all $i\equiv 1\pmod{p^m}$ with $i\geq 2$, there is a exact sequence
\[
0\to \ker{f^S_1}\to\coh{1}{G}{U^S_F}\to (\Z/p\Z)^t\to
\coker{f^S_1}\to\coker{f_i}\to\dual{H(E/F)}\to 0. 
\] 
\end{thm}
\begin{proof} Let $f_1'$ be the functorial homomorphism $\cls{S}{F}\to
  \Inv{G}{\cls{S}{E}}$. Since $E/F$ has degree $p$, we have
  $\ker{f^S_1}=\ker{f_1'}$ and $\coker{f^S_1}=\coker{f_1'}$.

The first part of the sequence is the well-known formula of Chevalley:
Let $E/F$ be a Galois extension of number fields. Let
$P^S_F=F^\times/\un{S}{F}$ and let $I^S_F$ be the group of
$S$-fractional ideals of $F$. By considering the natural map from the
sequence $1\to P^S_F \to I^S_F \to \cls{S}{F}\to 1$ to the
  corresponding sequence for $E$, by taking $G$-invariants and then
  applying the snake lemma, one obtains an exact sequence
\[
0\to \ker{f^S_1}\to (P^S_E)^G/P^S_F\to \oplus_{v\not\in
  S}\Z/e_v\Z\to \coker{f^S_1}\to
  \coh{1}{G}{P^S_E}\to 0.
\]
The surjectivity of the last map follows from the fact that $I^S_E$ is
a permutation $\Z[G]$-module, and thus $\coh{1}{G}{I^S_E}=0$ by
Shapiro's Lemma.  Hilbert's Theorem 90 gives  a natural isomorphism 
$(P^S_E)^G/P^S_F\cong \coh{1}{G}{\un{S}{E}}$. 

Now, by Hilbert's Theorem 90, 
\[
\coh{1}{G}{P^S_E}\cong\ker{\coh{2}{G}{\un{S}{E}}\to \coh{2}{G}{E^\times}}.
\]
Since $G$ is cyclic, the right-hand side is just
\[
\ker{\un{S}{F}/\norm{E}{F}(\un{S}{E})\to
  F^\times/\norm{E}{F}(E^\times)}=\frac{U^S_F\cap \norm{E}{F}(\units{E})}{\norm{E}{F}(U^S_E)}.
\]

However, by our assumptions on $E$, $F$ and $m$, we have
$\tatk{F}{1}=\un{S}{F}(\units{F})^p$,
$\tatk{E}{1}=\un{S}{E}(\units{E})^p$ and hence 
$\tatk{F}{i}=\un{S}{F}(\units{F})^p$ and
$\tatk{E}{i}=\un{S}{E}(\units{E})^p$ whenever $i\equiv 1\pmod{p^m}$
and thus
\[
\coker{f_i}\cong \tatk{F}{i}/\norm{E}{F}(\tatk{E}{i})\cdot(F^\times)^p
=\un{S}{F}(F^\times)^p/\norm{E}{F}(\un{S}{E})(F^\times)^p\cong \un{S}{F}/\norm{E}{F}(\un{S}{E})
\]
whenever $i\equiv 1\pmod{p^m}$.
Finally, by (the proof of) Theorem \ref{thm:main} above, the Kummer
pairing induces  a
natural isomorphism 
\[
\frac{U^S_F}{U^S_F\cap\norm{E}{F}{(\units{E})}}\cong \dual{H(E/F)}.
\]
\end{proof}
\begin{cor}\label{cor:coker}
Let $E/F$ be cyclic of degree $p$. Suppose that $E$ has only one
$p$-adic prime and that no non $p$-adic primes ramify in $E/F$. Let
$m=\mathrm{max}(m_F,m_E)$.
 Then 
\[
\coker{f_i}\cong \coker{f^S_1}\mbox{ for all }i\geq 2, \ i\equiv 1\pmod{p^m}.
\]
\end{cor}
\begin{rem}
If $E$ or $F$ have more than one $p$-adic prime then we need to
replace $U^S_F$ by $\gk{F}$, of course, and the $S$-class group
$\as{F}$ should be replaced by the \emph{logarithmic class group}
$\logcls{F}$. In these circumstances there is an analogous exact
sequence relating the cohomology of the logarithmic units to the
kernel and cokernel of the natural functorial homomorphism
\[
f_1:\logcls{F}\to\Inv{G}{\logcls{E}}. 
\]
For details of this sequence and of the logarithmic class group, see the article
of Jaulent, \cite{jaulent:log}. 
However, the cohomology of the group $\ldiv{E}$ of logarithmic
divisors is somewhat  more complicated than the cohomology of the group $I^S_E$ of
$S$-divisors. 
\end{rem}
However, in the particular case of a finite layer of the
$p$-cyclotomic extension of $F$ we have:
\begin{thm}\label{thm:cyclo}
Let $E=F_k$ for some $k\geq 2$.  Suppose that the Gross conjecture holds for $(E,p)$. Let
$p^n=[F_k:F]$ and let $m=\mathrm{max}(m_n(F),m_n(E))$.  Then 
\[
|\ker{f_i}|=|\ker{f_1}|\mbox{ for all } i\equiv 1\pmod{p^m}.
\]
\end{thm}
\begin{proof}
Let $G=\gal{E}{F}$.
By Corollary \ref{cor:f1} we have
\[
\coker{f_i}\cong \gk{F}/\norm{E}{F}(\gk{E})=\coh{2}{G}{\gk{E}}.
\]
Now use Theorem \ref{thm:jaul}.
\end{proof}
\begin{lem}\label{lem:cap}
Let $E/F$ be a cyclic extension with Galois group $G$ in which at least one prime ramifies
totally. Then $|\coker{f_1^S}|\geq |\ker{f_1^S}|$. 
\end{lem}
\begin{proof}
The hypothesis on ramification implies, via classfield theory, that
the norm map $\as{E}\to \as{F}$ is surjective. Consider the short
exact sequence (defining $K$) of $G$-modules  
\[
0\to K\to \as{E}\to \as{F}\to 0.
\]
Taking $G$-invariants gives the exact sequence
\[
0\to \coh{0}{G}{K}\to \as{E}^G\to \as{F}\to \coh{1}{G}{K}.
\]
Thus 
\[
|\as{E}^G|\geq |\as{F}|\cdot
 \frac{|\coh{0}{G}{K}|}{|\coh{1}{G}{K}|} = |\as{F}|
\]
since $G$ is cyclic and $K$ is a finite $G$-module.
\end{proof}
\begin{cor}
Suppose that $F$ has exactly one $p$-adic prime and that $E/F$ is
cyclic degree $p$ ramified at the $p$-adic prime  and at no other
prime. Let $m=\mathrm{max}(m_F,m_E)$. Then 
\[
|\ker{f_i}|\geq |\ker{f_1^S}|=|\coh{1}{G}{U^S_E}|\mbox{ for all } i\equiv
 1\pmod{p^m},i\geq 2.
\]
\end{cor}
\begin{proof} The first inequality follows from Lemma \ref{lem:cap}
 and Corollary \ref{cor:coker} above. The second equality follows from
 the proof of Theorem \ref{thm:coker}
\end{proof}
\begin{exa}\label{exa:cap}
We consider the following situation: the field $F$ is a CM-field with
one prime ideal dividing $p$. We will assume further that
$\psyl{\cl{F_+}}{p}=0$. 
Furthermore $E/F$ is cyclic degree $p$
extension of CM-fields which is unramified at all primes not dividing
$p$.  (It follows
therefore that $E_+/F_+$ - and hence $E/F$ - ramifies at the unique
$p$-adic prime.)  Finally, we will suppose that $\proots{E}=\proots{F}$;
i.e. $E/F$ is \emph{not} a $p$-cyclotomic extension.

In this situation $\h{F}=\h{E}=0$, so that $m_F=m_E=0$ (see Remark \ref{rem:h}).
Thus $|\ker{f_i}|\geq |\ker{f_1^S}|$ for all $i\geq 2$. 

Let $G=\gal{E}{F}$ and let $H$ be the group of order $2$ generated by
complex conjugation $J$. If $M$ is a $H$-module on which $2$ is
invertible, let 
\[
e_+=\frac{1+J}{2},\ e_-=\frac{1-J}{2}\ \in \edm{M}.
\] 
So $M=e_+(M)\oplus e_-(M)$.

Now $\ker{f_1^S} \cong \coh{1}{G}{U^S_E}$. Now
$\psyl{\cl{F_+}}{p}=e_+(\psyl{\cl{F}}{p})=0$ so that $e_+(A_S(F))=0$
  and hence $\ker{f_1^S}=e_-(\ker{f_1^S})$. It follows that 
$\coh{1}{G}{U^S_E}=e_-(\coh{1}{G}{U^S_E})$.

However, we observe the following: If $E/F$ is an odd-degree Galois
extension of CM-fields with the property that each $p$-adic prime is
stable under complex conjugation then
$e_-(\coh{1}{G}{U^S_E})=\coh{1}{G}{\mu(E)/\mu_{2^\infty}(E)}$. \textit{Proof:}
(cf. Theor\`eme 6 of \cite{jaulent:cap}) For an
abelian group $A$, let $A[1/2]= A\otimes\Z[1/2]$. Thus
$e_-(\coh{1}{G}{U^S_E})=
e_-(\coh{1}{G}{U^S_E}[1/2])=e_-(\coh{1}{G}{U^S_E[1/2]}) =
\coh{1}{G}{e_-(U^S_E[1/2])}$ (since the actions of $H$ and $G$
commute). Now if $u\in U^S_E[1/2]$, then the hypothesis on the
$p$-adic primes ensures that $(1-J)(u)\in U_E[1/2]$. Thus
$e_-(U^S_E[1/2])=e_-(U_E[1/2])$. If $u\otimes 1/2^n\in e_-(U_E[1/2])$,
then $u\bar{u}\otimes 1/2^n=1$ and hence $|u|=1$. The same holds for
all conjugates of $u$ since $E$ is a CM-field, and thus $u\otimes 1
\in \mu(E)[1/2]$.

Thus, in our situation,  $\ker{f^S_1}\cong
\coh{1}{G}{\proots{E}}=\coh{1}{G}{\proots{F}}=\hom{}{G}{\proots{F}}$ is a
group of order $p$ (and, in particular, $\as{F}\not= 0$). We conclude
that $|\ker{f_i}|\geq p$ for all $i\geq 2$.
\end{exa}
 
\begin{exa} A special case of the last example is the following:

Let $p$ be an irregular prime for which Vandiver's conjecture
holds. Let $F=\Q(\zeta_p)$. We are supposing that
$\psyl{\cl{F_+}}{p}=0$, but $\psyl{\cl{F}}{p}=\as{F}\not= 0$. 
Under these hypotheses, there exists a cyclic degree $p$ extension $E'/F_+$ which is
unramified outside $p$ and is not equal to the extension
$\Q(\zeta_{p^2})_+/F_+$ (Washington, \cite{wash:cf}, Proposition
10.13). 
Let $E=E'(\zeta_p)$ (and thus $E'=E_+$). So the hypotheses of Example
\ref{exa:cap} hold for $E/F$ and $\ker{f_i}\not= 0$ for all $i\geq 2$. 
\end{exa}

\textit{Acknowledgements:} I would like to thank A. Movahhedi for
directing my attention towards  the questions treated here and for his careful
reading of an earlier version of this paper. I would particularly like
to thank T. Nguyen Quang Do for very useful e-mail discussions about  Theorem
\ref{thm:nn}  (during which he showed me a proof of this theorem
different from the one given above) and for drawing my attention to the paper of Vauclair \cite{vauclair:cap}.




\end{document}